\documentclass[a4paper,12pt]{amsart}
\usepackage{verbatim}
\usepackage{amsmath}

\usepackage{tikz}

\newtheorem{theorem}{Theorem}[section]
\newtheorem{lemma}[theorem]{Lemma}
\newtheorem{proposition}[theorem]{Proposition}
\newtheorem{corollary}[theorem]{Corollary}

\theoremstyle{definition}
\newtheorem*{definition}{Definition}

\theoremstyle{remark}

\newtheorem*{remark*}{Remark}

\numberwithin{equation}{section}

\def\N{{\mathbb
N}}
\def\R{{\mathbb R}}

\newcommand{\xast}{x^{\ast}}

\newcommand{\yast}{y^{\ast}}
\newcommand{\Yast}{Y^{\ast}}

\newcommand{\Xast}{X^{\ast}}

\newcommand{\eps}{\varepsilon}

\begin{document}

\title{Stability of average roughness, octahedrality, and strong diameter two properties of Banach spaces\\ with respect to absolute sums}
% DIAMETER 2 PROPERTIES
%    Remove any unused author tags.
\date{26.04.2013}

%    author one information
\author{Rainis Haller}
%\address{Institute of Mathematics, University of Tartu, J.~Liivi 2, 50409 Tartu, Estonia}
\curraddr{}
%\email{rainis.haller@ut.ee, johann.langemets@ut.ee, mart.poldvere@ut.ee}
\thanks{This research was supported by institutional research funding IUT20-57 of the Estonian Ministry of Education and Research.}

%    author two information
\author{Johann Langemets}
%\address{Institute of Mathematics, University of Tartu, J.~Liivi 2, 50409 Tartu, Estonia}
\curraddr{}
%\email{johann.langemets@ut.ee}
\thanks{}

%    author three information
\author{Rihhard Nadel}
%\address{Institute of Mathematics, University of Tartu, J.~Liivi 2, 50409 Tartu, Estonia}
\curraddr{}
%\email{mart.poldvere@ut.ee}
\thanks{}

%    For articles to be published after 1 January 2010, you may use
%    the following version:
%\subjclass[2010]{Primary }


\date{}

\dedicatory{}

\begin{abstract}
We prove that, if Banach spaces $X$ and $Y$ are $\delta$-average rough, then their direct sum with respect to an absolute norm $N$ is $\delta/N(1,1)$-average rough. In particular, for octahedral $X$ and $Y$ and for $p$ in $(1,\infty)$ the space $X\oplus_p Y$ is $2^{1-1/p}$-average rough, which is in general optimal. Another consequence is that for any $\delta$ in $(1,2]$ there is a Banach space which is exactly $\delta$-average rough. We give a complete characterization when an absolute sum of two Banach spaces is octahedral or has the strong diameter 2 property. However, among all of the absolute sums, the diametral strong diameter 2 property is stable only for 1- and $\infty$-sums.
\end{abstract}

\maketitle
\section{Introduction}\label{sec: 0. Introduction}
A real Banach space $X$ is said to be \emph{octahedral} if, for every finite-dimensional subspace $E$ of $X$ and every $\varepsilon>0$, there is a norm one element $y\in X$ such that
\[
\|x+y\|\geq (1-\varepsilon)(\|x\|+\|y\|)\qquad \text{for all $x\in E$}.
\]
Octahedral Banach spaces were introduced by Godefroy and Maurey \cite{godefroy_normes} (see also \cite{godefroy_metric_1989}) in order to characterize Banach spaces containing an isomorphic copy of $\ell_1$. This notion has recently been useful in studying the diameter 2 properties (see \cite{becerra_guerrero_octahedral_2014}, \cite{becerra_guerrero_octahedral_2015}, \cite{haller_duality_2015}, and \cite{haller_rough_2017}).
It is known that  octahedrality is stable by taking $\ell_1$- or $\ell_\infty$-sums, and it is not stable by taking $\ell_p$-sums for $p\in(1,\infty)$ (see \cite[Proposition~3.12]{haller_duality_2015}). More precisely, for nontrivial Banach spaces $X$ and $Y$,
\begin{itemize}
\item if $X$ or $Y$ is octahedral, then $X\oplus_1 Y$ is octahedral; 
\item if $X$ and $Y$ are both octahedral, then  $X\oplus_\infty Y$ is octahedral;
\item $X\oplus_p Y$ is not octahedral for $p\in(1,\infty)$.
\end{itemize}
We extend these results quantitatively in two directions, instead of octahedral spaces we consider more general average rough spaces, and we also consider absolute normalized norm on direct sum.

Let $\delta>0$. A Banach space $X$ is said to be \emph{$\delta$-average rough} \cite{deville_dual_1988} if, whenever $n\in\N$ and $x_1,\dotsc,x_n \in S_X$,
\[
\limsup_{\|y\|\to0}\frac{1}{n}\sum_{i=1}^n \frac{\|x_i+y\|+\|x_i-y\|-2}{\|y\|}\geq\delta.
\]

Banach spaces which are $2$-average rough are exactly the octahedral ones (see \cite{becerra_guerrero_octahedral_2014}, \cite{deville_dual_1988}, and \cite{godefroy_metric_1989}).

We recall that a norm $N$ on $\mathbb R^2$ is called \emph{absolute} (see \cite{bonsall_numerical_1973}) if
\[
N(a,b)=N(|a|,|b|)\qquad\text{for all $(a,b)\in\mathbb R^2$}
\]
and \emph{normalized} if
\[
N(1,0)=N(0,1)=1.
\]
For example, the $\ell_p$-norm $\|\cdot\|_p$ is absolute and normalized for every $p\in[1,\infty]$. If $N$ is an absolute normalized norm on $\mathbb R^2$ (see \cite[Lemmata~21.1 and 21.2]{bonsall_numerical_1973}), then 

\begin{itemize}
\item $\|(a,b)\|_\infty\leq N(a,b)\leq \|(a,b)\|_1$ for all $(a,b)\in\mathbb R^2$;
\item if $(a,b),(c,d)\in\mathbb R^2$ with $|a|\leq |c|\quad \text{and}\quad |b|\leq |d|,$ then \[N(a,b)\leq N(c,d);\]
\item the dual norm $N^\ast$ on $\mathbb \R^2$ defined by 
\[
N^\ast(c,d)=\max_{N(a,b)\leq 1}(|ac|+|bd|) \qquad\text{for all $(c,d)\in\mathbb R^2$}
\]
is also absolute and normalized. Note that $(N^\ast)^\ast=N$.
\end{itemize}

If $X$ and $Y$ are Banach spaces and $N$ is an absolute normalized norm on $\R^2$, then we denote by $X\oplus_N Y$ the product space $X\times Y$ with respect to the norm
\[
\|(x,y)\|_N=N(\|x\|,\|y\|) \qquad\text{for all $x\in X$ and $y\in Y$}.
\]
In the special case where $N$ is the $\ell_p$-norm, we write $X\oplus_p Y$.
Note that $(X\oplus_N Y)^\ast=\Xast\oplus_{N^\ast} \Yast$.

By a \emph{slice} of $B_X$ we mean a set of the form
\[
S(B_X, x^*,\alpha)=\{x\in B_X\colon x^*(x)>1-\alpha\},
\]
where $x^*\in S_{X^*}$ and $\alpha>0$. A \emph{convex combination of slices} is a set of the form $\sum_{i=1}^n \lambda_i S_i$, where $n\in\mathbb N$, $\lambda_1,\dotsc,\lambda_n\geq 0$ with $\sum_{i=1}^n\lambda_i=1$, and $S_1,\dotsc,S_n$ are slices of $B_X$. 

A dual characterization of $\delta$-average roughness is well known. The dual space $\Xast$ is $\delta$-average rough if and only if the diameter of every convex combination of  slices of $B_{X}$ is greater than or equal to $\delta$ \cite[Theorem~2]{deville_dual_1988}. In particular, $\Xast$ is octahedral 
if and only if the diameter of every convex combination of  slices of $B_{X}$ is $2$ (see also \cite{becerra_guerrero_octahedral_2014},  \cite{godefroy_metric_1989}, and \cite{haller_duality_2015}). According to \cite{abrahamsen_remarks_2013}, the latter extreme property of a Banach space $X$ is known as the \emph{strong diameter 2 property}. An important class of Banach spaces with the strong diameter 2 property and which are octahedral are the Daugavet spaces (see \cite{abrahamsen_remarks_2013} and \cite{becerra_guerrero_octahedral_2014}). 

In \cite{bilik_narrow_2005}, it is proved that the only absolute sums which preserve the Daugavet property are the $\ell_1$- and $\ell_\infty$-sum. Surprisingly, there are  many absolute norms which preserve octahedrality and the strong diameter 2 property (see Section~\ref{sec: OH and SD2P}).

Recently, Becerra~Guerrero, {L}\'{o}pez-{P}\'{e}rez, and Rueda~Zoca introduced a sharper version of the strong diameter 2 property (see \cite{becerra_guerrero_diametral_2015}). A Banach space $X$ has the
\emph{diametral strong diameter 2 property} if for every convex combination $C$ of relatively weakly open subsets of $B_X$, for every $x\in C$ and $\varepsilon>0$ there is a $y\in C$ such that
\[
\|x-y\|>1+\|x\|-\varepsilon.
\]

By \cite{becerra_guerrero_diametral_2015}, Daugavet spaces have the diametral strong diameter 2 property and the diametral strong diameter 2 property implies the strong diameter 2 property. The  Banach space $c_0$ is an example of a space with the strong diameter 2 property and failing the diametral strong diameter 2 property. As far as the authors know it is an open question posed in \cite{becerra_guerrero_diametral_2015} whether there is a Banach space with the diametral strong diameter 2 property and failing the Daugavet property. Our preliminary idea to attack this question was to check whether besides $\ell_1$- and $\ell_\infty$-norm (see \cite{becerra_guerrero_diametral_2015} and \cite{haller_diametral_2016}) there are more absolute norms which preserve the diametral strong diameter 2 property. However, there are none (see Section~\ref{sec: OH and SD2P}).

We now describe the contents of this paper. In Section~\ref{sec: average rough}, we prove (see Theorem~\ref{thm: abs sum delta rough}) that for $\delta$-average rough Banach spaces $X$ and $Y$ their absolute sum $X\oplus_N Y$ is $\gamma\delta$-average rough, where $\gamma>0$ is such that $\gamma N(\cdot)\leq \Vert \cdot\Vert_{\infty}$. In particular, we get that, for $1<p<\infty$, the $\ell_p$-sum $X\oplus_p Y$ of octahedral Banach spaces $X$ and $Y$ is $2^{1-1/p}$-average rough (see Corollary~\ref{cor: p-sum of OH spaces is 2^{1/q}-average rough}). Moreover, this number $2^{1-1/p}$ is in general the largest possible one (see Proposition~\ref{prop: l1 p-sum is not average rough}). As a consequence, we obtain that for any $\delta\in (1,2]$ there is a Banach space which is exactly $\delta$-average rough (see Theorem~\ref{thm: exactly delta-average rough}). We end this section by describing when the $\delta$-average roughness passes down from the absolute sum to one of the factors (see Proposition~\ref{prop: roughness from sum to factor}).

In Section~\ref{sec: OH and SD2P}, we characterize those absolute norms for which the direct sum of two octahedral Banach spaces is octahedral (see Theorem~\ref{thm: absolute sum is OH}). As a consequence, we can characterize those absolute norms for which the direct sum of two separable Banach spaces with the almost Daugavet property has the almost Daugavet property (see Corollary~\ref{cor: almost Daugavet}). By duality, we can characterize the absolute norms which preserve the strong diameter 2 property (see Theorem~\ref{thm: absolute sum has SD2P}).  We end this section by proving that, similarly to the Daugavet property, among all of the absolute norms the diametral strong diameter 2 property is stable only for $\ell_1$- and $\ell_\infty$-sums (see Corollary~\ref{cor: abs sum DSD2P}).

\section{Average roughness of absolute sums}\label{sec: average rough}

We begin by pointing out some equivalent but sometimes more convenient formulations of average roughness, which are easily derived from the definition.

\begin{proposition}\label{prop: reformulations of rough}
	Let $X$ be a Banach space and $\delta>0$. The following assertions are equivalent:
	\begin{enumerate}\label{KdK krit}
		\item[\rm (i)] $X$ is $\delta$-average rough;
		\item[\rm (ii)] whenever $n\in \mathbb{N}$,  $x_{1},\dotsc,x_{n}\in X$,  and $\varepsilon>0$ there is a $y\in X$ such that $\Vert y \Vert\leq \varepsilon$ and
		\begin{displaymath}
		\frac{1}{n}\sum_{i=1}^{n}\Big( \Vert x_{i}+y\Vert +\Vert x_{i}-y\Vert \Big)> (\delta-\varepsilon)\Vert y \Vert +\frac{2}{n}\sum_{i=1}^{n}\Vert x_{i}\Vert;
		\end{displaymath}
		\item[\rm (iii)] whenever $n\in \mathbb{N}$, $x_{1},\ldots,x_{n}\in S_{X}$, and $\varepsilon>0$ there is a $y\in X$ such that $\Vert y \Vert\leq \varepsilon$ and
		\begin{displaymath}
		\frac{1}{n}\sum_{i=1}^{n}\Big( \Vert x_{i}+y\Vert +\Vert x_{i}-y\Vert \Big)> (\delta-\varepsilon)\Vert y \Vert +2;
		\end{displaymath}
%		\item[\rm (iv)] for every $x_{1},\ldots,x_{n}\in X$ $ (n\in \mathbb{N}) $ and $\varepsilon>0$ there is a $y\in X$ such that $\Vert y \Vert=\varepsilon$ and
%		\begin{displaymath}
%		\frac{1}{n}\sum_{i=1}^{n} \Big(\Vert x_{i}+y\Vert +\Vert x_{i}-y\Vert\Big) > (\delta-\varepsilon)\Vert y \Vert +\frac{2}{n}\sum_{i=1}^{n}\Vert x_{i}\Vert;
%		\end{displaymath}
%		\item[\rm (v)] for every $x_{1},\dotsc,x_{n}\in S_{X}$ $ (n\in \mathbb{N}) $ and $\varepsilon>0$ there is a $y\in X$ such that $\Vert y \Vert=\varepsilon$ and
%		\begin{displaymath}
%		\frac{1}{n}\sum_{i=1}^{n}\Big( \Vert x_{i}+y\Vert +\Vert x_{i}-y\Vert \Big)> (\delta-\varepsilon)\Vert y \Vert +2.
%		\end{displaymath}
	\end{enumerate}
\end{proposition}

\begin{remark*}\label{rem: 1/n = lambda}
The equivalences in Proposition~\ref{prop: reformulations of rough} remain true if one of the following holds
\begin{itemize}
\item[(a)]one replaces $\frac{1}{n}\sum_{i=1}^n$ with $\sum_{i=1}^n\lambda_i$, where $\lambda_i>0$ and $\sum_{i=1}^n\lambda_i=1$;
\item[(b)] one replaces $\|y\|\leq \varepsilon$ with $\|y\|=\varepsilon$.
\end{itemize}
\end{remark*}

The $\ell_1$-sum of two Banach spaces inherits its $\delta$-average roughness from one of the factors.
\begin{proposition}
Let $X$ and $Y$ be Banach spaces. If $X$ or $Y$ is $\delta$-average rough for some $\delta>0$, then $X\oplus_1 Y$ is also $\delta$-average rough.  
\end{proposition}
\begin{proof}
We consider only the case where $X$ is $\delta$-average rough. The case where $Y$ is $\delta$-average rough is similar. We will prove that $Z=X\oplus_1 Y$ is $\delta$-average rough. Let $z_{1}=(x_{1},y_{1}),\dotsc,z_{n}=(x_{n},y_{n})\in S_{Z}$ and $\varepsilon>0$. By Proposition~\ref{prop: reformulations of rough}, it suffices to show that there exists $z=(x,y)\in Z$ such that $\Vert z \Vert_1 =  \varepsilon$ and
	\begin{displaymath}
	\frac{1}{n}\sum_{i=1}^{n}\Big( \Vert z_{i}+z\Vert_1 + \Vert z_{i}-z\Vert_1 \Big) \geq (\delta-\varepsilon)\Vert z \Vert_1 +2.
	\end{displaymath}
	Since $X$ is $\delta$-average rough, there is an $x\in X$  such that $\Vert x \Vert = \varepsilon$ and
	\begin{displaymath}
	\sum_{i=1}^{n} \frac1n\Big( \Vert x_{i}+x\Vert + \Vert x_{i}-x\Vert \Big) \geq (\delta-\varepsilon)\Vert x \Vert +\frac2n\sum_{i=1}^{n} \Vert x_{i}\Vert.
	\end{displaymath}
It follows that, for $z=(x,0)$ we have $\|z\|_1=\|x\|=\varepsilon$, and 
\begin{align*}
&\frac{1}{n}\sum_{i=1}^{n}\Big( \Vert z_{i}+z\Vert_1 + \Vert z_{i}-z\Vert_1 \Big)\\
&=\frac{1}{n}\sum_{i=1}^{n}\Big( \Vert x_{i}+x\Vert+\|y_i\| + \Vert x_{i}-x\Vert +\|y_i\| \Big)\\
&\geq (\delta-\varepsilon)\Vert x \Vert +\frac2n\sum_{i=1}^{n} \Vert x_{i}\Vert+\frac2n\sum_{i=1}^{n} \Vert y_{i}\Vert\\
&=(\delta-\varepsilon)\Vert z \Vert_1+2.
\end{align*}
\end{proof}

\begin{corollary}[see {\cite[Proposition~3.12]{haller_duality_2015}}]
If $X$ or $Y$ is octahedral, then $X\oplus_1 Y$ is also octahedral.
\end{corollary}

 The following theorem is one of the main results in this section.
 
\begin{theorem}\label{thm: abs sum delta rough}
	Let $X$ and $Y$ be Banach spaces, $N$ an absolute normalized norm on $\R^2$, and $\gamma>0$ be such that $\Vert \cdot\Vert_{\infty} \geq \gamma N(\cdot)$. If $X$ and $Y$ are $\delta$-average rough for some $\delta>0$, then $X\oplus_N Y$ is $\gamma\delta$-average rough.
\end{theorem}

\begin{proof}
	Assume that $X$ and $Y$ are $\delta$-average rough. We will prove that $Z=X\oplus_N Y$ is $\gamma\delta$-average rough. Let $z_{1}=(x_{1},y_{1}),\dotsc,z_{n}=(x_{n},y_{n})\in S_{Z}$ and $\varepsilon>0$. By Proposition~\ref{prop: reformulations of rough}, it suffices to show that there exists $z=(x,y)\in Z$ such that $\Vert z \Vert_N = \varepsilon N(1,1)$ and
	\begin{displaymath}
	\frac{1}{n}\sum_{i=1}^{n}\Big( \Vert z_{i}+z\Vert_N + \Vert z_{i}-z\Vert_N \Big) \geq (\delta-\varepsilon)\gamma\Vert z \Vert_N +2.
	\end{displaymath}
	Choose $c_i, d_i\geq 0$ such that $N^\ast(c_{i},d_{i})=1$ and $c_{i}\Vert x_{i}\Vert+d_{i}\Vert y_{i}\Vert=1$. Denote by 
	\begin{displaymath}
	c= \frac{1}{n}\sum_{i=1}^{n}c_{i} \qquad\text{and}\qquad d= \frac{1}{n}\sum_{i=1}^{n}d_{i}.
	\end{displaymath}
  Note that $c+d\geq 1$, because $c_i+d_i\geq N^\ast(c_i,d_i)=1$.  Consider first the case where $c\neq 0$ and $d\neq 0$. Denote by
	\begin{displaymath}
	\mu_{i}=\frac{1}{n} \frac{c_{i}}{c} \qquad\text{and}\qquad \nu_{i}=\frac{1}{n} \frac{d_{i}}{d}.
	\end{displaymath}
	Observe that $\mu_{1}+\dotsb+\mu_{n}=\nu_{1}+\dotsb+\nu_{n}=1$. Since $X$ and $Y$ are $\delta$-average rough, by Proposition~\ref{prop: reformulations of rough}, there are $x\in X$ and $y\in Y$ such that $\Vert x \Vert= \Vert y \Vert = \varepsilon$ and
	\begin{displaymath}
	\sum_{i=1}^{n} \mu_{i}\Big( \Vert x_{i}+x\Vert + \Vert x_{i}-x\Vert \Big) \geq (\delta-\varepsilon)\Vert x \Vert +2\sum_{i=1}^{n} \mu_{i}\Vert x_{i}\Vert
	\end{displaymath}
	and
	\begin{displaymath}
	\sum_{i=1}^{n} \nu_{i}\Big( \Vert y_{i}+y\Vert + \Vert y_{i}-y\Vert \Big) \geq (\delta-\varepsilon)\Vert y \Vert +2\sum_{i=1}^{n} \nu_{i}\Vert y_{i}\Vert.
	\end{displaymath}
It follows that, for $z=(x,y)$ we have $\|z\|_N=\varepsilon N(1,1)$, and
	\begin{align*}
		&\frac{1}{n} \sum_{i=1}^{n} \Big(  \Vert z_{i}+z\Vert_N + \Vert z_{i}-z\Vert_N \Big) \\
		&\qquad \geq \frac{1}{n}\sum_{i=1}^{n} N \big( \Vert x_{i}+x\Vert +\Vert x_{i}-x\Vert,\Vert y_{i}+y\Vert +\Vert y_{i}-y\Vert \big) \\
		&\qquad \geq \frac{1}{n} \sum_{i=1}^{n} \Big( c_{i} \big(\Vert x_{i}+x\Vert +\Vert x_{i}-x\Vert\big)+d_{i} \big(\Vert y_{i}+y\Vert +\Vert y_{i}-y\Vert\big)\Big) \\
		&\qquad = c \sum_{i=1}^{n} \mu_{i} \big(\Vert x_{i}+x\Vert +\Vert x_{i}-x\Vert\big)+d \sum_{i=1}^{n} \nu_{i} \big(\Vert y_{i}+y\Vert +\Vert y_{i}-y\Vert\big)\\
		&\qquad \geq c \Big((\delta-\varepsilon)\Vert x \Vert +2\sum_{i=1}^{n} \mu_{i}\Vert x_{i}\Vert\Big)+d\Big( (\delta-\varepsilon)\Vert y \Vert +2\sum_{i=1}^{n} \nu_{i}\Vert y_{i}\Vert\Big)\\
		&\qquad = (\delta-\varepsilon)(c\Vert x \Vert +d\Vert y \Vert)+\frac{2}{n}\sum_{i=1}^{n} (c_{i}\Vert x_{i}\Vert + d_{i}\Vert y_{i}\Vert)\\
		&\qquad = (\delta-\varepsilon)(c+d)\max\{\|x\|, \|y\|\} +2 \\
		&\qquad \geq (\delta-\varepsilon) \gamma N(\Vert x \Vert,\Vert y \Vert)+2\\
		&\qquad =(\delta-\varepsilon)\gamma\Vert z \Vert_{N} +2.
	\end{align*}
	
	Consider now the case where $c=0$, which means that $c_{i}=0$ and $d_i=1$ for all $i\in\{1,\dotsc,n\}$. This implies that $\|y_i\|=1$ for all $i\in\{1,\dotsc,n\}$. Since $Y$ is $\delta$-average rough, by Proposition~\ref{prop: reformulations of rough}, there exists a $y\in Y$ such that $\|y\|=\varepsilon N(1,1)$ and  
	\begin{displaymath}
	\sum_{i=1}^{n} \frac1n\Big( \Vert y_{i}+y\Vert + \Vert y_{i}-y\Vert \Big) \geq (\delta-\varepsilon)\Vert y \Vert +2.
	\end{displaymath}
Therefore, for $z=(0,y)$ we have $\|z\|_N=\|y\|=\varepsilon N(1,1)$, and
	\begin{align*}
	&\frac{1}{n} \sum_{i=1}^{n} \Big(  \Vert z_{i}+z\Vert_N + \Vert z_{i}-z\Vert_N \Big) \\
	&\qquad \geq \frac{1}{n} \sum_{i=1}^{n}  \Big(\Vert y_{i}+y\Vert +\Vert y_{i}-y\Vert\Big) \\
	&\qquad \geq (\delta-\varepsilon)\Vert y \Vert +2 \\
	&\qquad \geq(\delta-\varepsilon)\gamma\Vert z \Vert_{N} +2.
	\end{align*}
	
	The case where $d=0$ is similar to the case $c=0$.	We have thus proved that $X\oplus_{N}Y$ is $\gamma\delta$-average rough.
\end{proof}

In particular, Theorem~\ref{thm: abs sum delta rough} applies to $\ell_p$-norms.

\begin{corollary}
If Banach spaces $X$ and $Y$ are $\delta$-average rough for some $\delta>0$, then
	\begin{itemize}
		\item[\rm (a)]  $X\oplus_{\infty}Y$ is $\delta$-average rough;
		\item[\rm (b)] $X\oplus_{p}Y$ is $2^{-1/p}\delta$-average rough for $1<p<\infty$.
	\end{itemize}
\end{corollary}

\begin{corollary}\label{cor: p-sum of OH spaces is 2^{1/q}-average rough}
If Banach spaces $X$ and $Y$ are octahedral and $1<p<\infty$, then $X\oplus_{p}Y$ is $2^{1-1/p}$-average rough.
\end{corollary}

%In \cite[Proposition~2.1]{becerra_guerrero_extreme_2015} it was shown that $\ell_1\oplus_p \ell_1$ is $1$-average for any $p\geq1$.

In Corollary~\ref{cor: p-sum of OH spaces is 2^{1/q}-average rough}, we saw that if $X$ and $Y$ are octahedral and $1<p<\infty$, then $X\oplus_p Y$ is $2^{1-1/p}$-average rough. We will now prove that in general $2^{1-1/p}$ is the largest possible number.

\begin{proposition}\label{prop: l1 p-sum is not average rough}
Let $X$ and $Y$ be Banach spaces and $1<p<\infty$. Then $X\oplus_{p}Y$ is not $\delta$-average rough for any $\delta>2^{1-1/p}$.
\end{proposition}

\begin{proof}
We will prove that $Z=X\oplus_{p}Y$ is not $\delta$-average rough for any $\delta>2^{1-1/p}$.
	Consider the elements $z_{1}=(x_0,0)$ and $z_{2}=(0,y_0)$ in $Z$, where $x_0\in S_X$ and $y_0\in S_Y$. It suffices to show that there is a function $f:(0,\infty)\rightarrow \mathbb{R}$ such that $f(\varepsilon)\rightarrow 0$, when $\varepsilon\rightarrow 0$, and that for every $\varepsilon>0$ and $z\in Z$, where $\Vert z \Vert =\varepsilon$, 
	\begin{displaymath}
	\frac{1}{2}\Big( \Vert z_{1}+z\Vert_{p}+\Vert z_{1}-z\Vert_{p}+\Vert z_{2}+z\Vert_{p}+\Vert z_{2}-z\Vert_{p}\Big)  \leq \Big( 2^{1-1/p}+f(\varepsilon)\Big) \Vert z \Vert_{p} +2.
	\end{displaymath}
	Let $\varepsilon\in (0,1)$. Let $z=(x,y)\in Z$ be such that $\Vert z \Vert_{p} =\varepsilon$. 
	By Maclaurin's formula,
	\begin{displaymath}
	(1+\Vert x \Vert)^{p}=1+p\Vert x \Vert+\frac{p(p-1)}{2}(1+\xi)^{p-2}\Vert x \Vert^{2},
	\end{displaymath}
	for some $\xi \in (0,\Vert x \Vert)$.
	Observe that
	\begin{equation}\label{eq: general estimate}
	\begin{aligned}
	\Vert z_{1}\pm z\Vert_{p}^{p}&=\Vert x_0\pm x \Vert^{p}+\Vert y \Vert^{p} \leq ( 1+\Vert x \Vert)^{p}+\Vert y \Vert^{p}\\
	&= 1+p\Vert x \Vert +\frac{p(p-1)(1+\xi)^{p-2}}{2}\Vert x \Vert^{2}+\Vert y \Vert^{p}.
	\end{aligned}
	\end{equation}
	We continue by considering the cases $1<p\leq 2$ and $p>2$ separately. In both cases we will use the generalized Bernoulli's inequality, which says that for any $t\geq 0$ we have $	(1+t)^{1/p}\leq 1+t/p$. 
	
\textbf{Case I}. Assume that $1<p\leq 2$.
Since $\xi \in (0,\Vert x \Vert)$, we have
	\begin{displaymath}
	(1+\xi)^{p-2} \leq (1+0)^{p-2} = 1.
	\end{displaymath}
Combining the estimate (\ref{eq: general estimate}) with Bernoulli's inequality we get
	\begin{displaymath}
	\begin{aligned}
	\Vert z_{1}\pm z\Vert_{p}&\leq \Big(1+p\Vert x \Vert +\frac{p(p-1)}{2}\Vert x \Vert^{2}+\Vert y \Vert^{p}\Big)^{1/p}\\
	&\leq 1+\Vert x \Vert +\frac{p-1}{2}\Vert x \Vert^{2}+\frac{\Vert y \Vert^{p}}{p}.
	\end{aligned}
	\end{displaymath}
	Similarly, we obtain
	\begin{displaymath}
	\begin{aligned}
	\Vert z_{2}\pm z\Vert_{p}&\leq 1+\frac{\Vert x \Vert^{p}}{p} +\frac{p-1}{2}\Vert y \Vert^{2}+\Vert y \Vert.
	\end{aligned}
	\end{displaymath}
Therefore
	\begin{displaymath}
	\begin{aligned}
	&\frac{1}{2}\Big(\Vert z_{1}+z\Vert_{p} + \Vert z_{1}-z\Vert_{p}+ \Vert z_{2}+z\Vert_{p} + \Vert z_{2}-z\Vert_{p}\Big)\\
	&\leq  \bigg( 1+\Vert x \Vert +\frac{p-1}{2}\Vert x \Vert^{2}+\frac{\Vert y \Vert^{p}}{p}\bigg)+\bigg(1+\frac{\Vert x \Vert^{p}}{p} +\frac{p-1}{2}\Vert y \Vert^{2}+\Vert y \Vert\bigg)\\
	& = 2+\Vert x \Vert +\Vert y \Vert+\frac{p-1}{2}(\Vert x\Vert ^{2}+\Vert y\Vert ^{2} )+\frac{1}{p} (\Vert x \Vert^{p}+\Vert y \Vert^{p})\\
	&\leq 2+2^{1-1/p}\Vert (x,y)\Vert_{p}+\frac{p-1}{2}\varepsilon^{2}+\frac{\varepsilon^{p}}{p}\\
	&= 2+\left( 2^{1-1/p}+\frac{p-1}{2}\varepsilon+\frac{\varepsilon^{p-1}}{p}\right)  \Vert z \Vert_p.
	\end{aligned}
	\end{displaymath}
	Thus, for $1< p \leq2$, we can take  \begin{displaymath}
	f(\varepsilon)= \frac{p-1}{2}\varepsilon+\frac{\varepsilon^{p-1}}{p}.
	\end{displaymath}
	
	\textbf{Case II}. Assume that $p>2$.
Since $\xi \in (0,\Vert x \Vert)$ and $\Vert x \Vert \leq \varepsilon<1$, we have
	\begin{displaymath}
	(1+\xi)^{p-2}\leq(1+\Vert x \Vert)^{p-2} \leq (1+\varepsilon)^{p-2} < 2^{p-2}.
	\end{displaymath}
Combining this estimate with (\ref{eq: general estimate}) and Bernoulli's inequality, we get
	\begin{displaymath}
	\begin{aligned}
	\Vert z_{1}\pm z\Vert_{p}&\leq \big(1+p\Vert x \Vert +p(p-1)2^{p-3}\Vert x \Vert^{2}+\Vert y \Vert^{p}\big)^{1/p}\\
	&\leq 1+\Vert x \Vert +(p-1)2^{p-3}\Vert x \Vert^{2}+\frac{\Vert y \Vert^{p}}{p}.
	\end{aligned}
	\end{displaymath}
	Similarly, we obtain
	\begin{displaymath}
	\begin{aligned}
	\Vert z_{2}\pm z\Vert_{p}&\leq (\Vert x \Vert^{p}+1+p\Vert y \Vert +p(p-1)2^{p-3}\varepsilon^{2})^{1/p}\\
	&\leq 1+\frac{\Vert x \Vert^{p}}{p} +(p-1)2^{p-3}\Vert y \Vert^{2}+\Vert y \Vert.
	\end{aligned}
	\end{displaymath}
Therefore
	\begin{displaymath}
	\begin{aligned}
	&\frac{1}{2}\Big(\Vert z_{1}+z\Vert_{p} + \Vert z_{1}-z\Vert_{p}+ \Vert z_{2}+z\Vert_{p} + \Vert z_{2}-z\Vert_{p}\Big)\\
	&\leq \bigg( 1+\Vert x \Vert +(p-1)2^{p-3}\varepsilon^{2}+\frac{\Vert y \Vert^{p}}{p}\bigg)\\
	&\qquad +\bigg(1+\frac{\Vert x \Vert^{p}}{p} +(p-1)2^{p-3}\varepsilon^{2}+\Vert y \Vert\bigg)\\
	& =2+ (\Vert x \Vert+\Vert y \Vert)+(p-1)2^{p-3}(\Vert x\Vert ^{2}+\Vert y\Vert ^{2} )+\frac{1}{p}(\Vert x \Vert^{p}+\Vert y \Vert^{p}) \\
	&\leq 2+2^{1-1/p}\Vert (x,y)\Vert_{p}+(p-1)2^{p-3}2^{1-2/p}\varepsilon^{2} +\frac{\varepsilon^{p}}{p}\\
	&=\left( 2^{1-1/p}+(p-1)2^{p-2/p-2}\varepsilon+\frac{\varepsilon^{p-1}}{p}\right) \Vert z \Vert+2.
	\end{aligned}
	\end{displaymath}
	Thus, for $p>2$, we can take 
	\begin{displaymath}
	f(\varepsilon)=(p-1)2^{p-2/p-2}\varepsilon+\frac{\varepsilon^{p-1}}{p}.
	\end{displaymath}
	
Hence  $X\oplus_{p}Y$ is not $\delta$-average rough for any $\delta>2^{1-1/p}$.
\end{proof}

Now we are ready to show that for any $\delta\in(1,2]$ there is a Banach space which is exactly $\delta$-average rough.

\begin{theorem}\label{thm: exactly delta-average rough}
	For any $\delta\in (1,2]$ there is a dual Banach space, which is $\delta$-average rough and is not $\gamma$-average rough for any $\gamma>\delta$.
\end{theorem}
\begin{proof}
	If $\delta=2$, then we can take $\ell_{1}$. If $\delta\in (1,2)$, then there is a $q\in (1,\infty)$ such that $\delta=2^{1/q}$. Let $p\in (1,\infty)$ be such that $1/p+1/q=1$.
	Since $\ell_{1}$ is octahedral, then by Corollary~\ref{cor: p-sum of OH spaces is 2^{1/q}-average rough} and Proposition~\ref{prop: l1 p-sum is not average rough} the Banach space $\ell_{1}\oplus_{p}\ell_{1}$ is $\delta$-average rough and is not $\gamma$-average rough for any $\gamma>\delta$.
\end{proof}

\begin{remark*}
We do not know whether a similar result to Theorem~\ref{thm: exactly delta-average rough} holds for $\delta\in (0,1]$.
\end{remark*}

Theorem~\ref{thm: exactly delta-average rough} and the dual caharacterization of $\delta$-average rough norms (see \cite[Theorem~2]{deville_dual_1988}) immediately implies the following.

\begin{corollary}
For any $\delta\in (1,2]$ there is a Banach space in which the minimal diameter of convex combination of slices is $\delta$.% greater than or equal to $\delta$ and there is a convex combination of slices with diameter exactly equal to $\delta$.
\end{corollary}

We end this section by describing when the $\delta$-average roughness passes down from the absolute sum to one of the factors. Our results are inspired by \cite[Proposition~2.5]{acosta_stability_2015}.

The following lemma is easily verified from the definitions.

\begin{lemma}\label{lemma: extreme point is an strongly exposed point}
	Let $N$ be an absolute normalized norm on $\mathbb{R}^{2}$ such that $(1,0)$ is an extreme point of the unit ball $B_{(\mathbb{R}^2, N)}$. Then $(1,0)$ is a strongly exposed point of $B_{(\mathbb{R}^2, N)}$, which is strongly exposed by the functional $(1,0)\in B_{(\mathbb{R}^2, N^\ast)} $. In particular, for every $\varepsilon>0$ there is a $\gamma>0$ such that, whenever $(a,b)\in B_{(\mathbb{R}^2, N)}$ and $a>1-\gamma$, then $|b|<\varepsilon$.
\end{lemma}

\begin{proposition}\label{prop: roughness from sum to factor}
Let $X$ and $Y$ be Banach spaces and $N$ an absolute normalized norm on $\R^2$ such that $(1,0)$ is an extreme point of $B_{(\mathbb{R}^2, N^\ast)}$.  If $X\oplus_{N}Y$ is $\delta$-average rough for some $\delta>0$, then $X$ is $\delta$-average rough.
\end{proposition}

\begin{proof}
Assume that $Z=X\oplus_{N}Y$ is $\delta$-average rough. Let $x_1,\dots, x_n\in S_X$ and $\varepsilon\in(0,\delta)$. We will show that there is a $u\in X$ such that $\|u\|\leq \varepsilon$ and
\begin{displaymath}
		\frac{1}{n}\sum_{i=1}^{n}\Big( \Vert x_{i}+u\Vert +\Vert x_{i}-u\Vert \Big)> \Big(\delta-\varepsilon\Big)\Vert u \Vert +2.
		\end{displaymath}
By Lemma~\ref{lemma: extreme point is an strongly exposed point}, there is a $\gamma\in(0,\frac{2\varepsilon}{3})$ such that, whenever $N^\ast(a,b)\leq 1$ and $a>1-\gamma$, then $|b|<\frac{\varepsilon}{3}$.

Consider $(x_i,0)\in S_{Z}$. Since $Z$ is $\delta$-average rough, there is a $z=(u,v)\in Z$ such that $\|z\|_N=\frac{\gamma}{2n}$ and
\begin{displaymath}
		\frac{1}{n}\sum_{i=1}^{n}\Big( \|(x_i,0)+(u,v)\|_N +\|(x_i,0)-(u,v)\|_N \Big)> (\delta-\gamma/2)\Vert z \Vert_N +2.
		\end{displaymath}
Choose $a_i,b_i, c_i,d_i\geq 0$ with $N^\ast(a_i,b_i)=N^\ast(c_i,d_i)=1$ such that
\begin{displaymath}
 a_i\Vert x_{i}+u\Vert+b_i\|v\|=N(\|x_i+u\|, \|v\|)
		\end{displaymath}
and
\begin{displaymath}
 c_i\Vert x_{i}-u\Vert+d_i\|v\|=N(\|x_i-u\|, \|v\|).
		\end{displaymath}
Then we have 		
\begin{displaymath}
		\frac{1}{n}\sum_{i=1}^{n}\Big( a_i(\Vert x_{i}\|+\|u\|)+b_i\|v\| +c_i(\Vert x_{i}\Vert+\|u\|)+d_i\|v\| \Big)> (\delta-\gamma/2)\Vert z \Vert_N +2,
		\end{displaymath}
which implies that
\[
\frac{a_i}{n}+\frac{n-1}{n}+1+2\|z\|_N>(\delta-\gamma/2)\Vert z \Vert_N +2.
\]
It follows that $a_i>1-\gamma$ and hence $b_i<\varepsilon/3$ for all $i\in\{1,\dots,n\}$. Similarly, one obtains that $c_i>1-\gamma$ and $d_i<\varepsilon/3$ for all $i\in\{1,\dots,n\}$.

Therefore
\begin{align*}
&\frac{1}{n}\sum_{i=1}^{n}\Big( \Vert x_{i}+u\Vert +\Vert x_{i}-u\Vert \Big)\\
&\geq\frac{1}{n}\sum_{i=1}^{n}\Big( a_i\|x_i+u\|\pm b_i\|v\| + c_i\|x_i-u\|\pm d_i\|v\| \Big)\\
&>(\delta-\gamma/2)\Vert z \Vert_N+2-2\frac{\varepsilon}{3}\|z\|_N\\
&=(\delta-\gamma/2-2\frac{\varepsilon}{3})\|z\|_N+2\\
&>(\delta-\varepsilon)\|u\|+2.
\end{align*}

\end{proof}

\begin{remark*}
One can prove similarly to Proposition~\ref{prop: roughness from sum to factor} that, if $N$ is an absolute normalized norm on $\R^2$ such that $(0,1)$ is an extreme point of $B_{(\mathbb{R}^2, N^\ast)}$ and $X\oplus_{N}Y$ is $\delta$-average rough for some $\delta>0$, then $Y$ is $\delta$-average rough.
\end{remark*}

\begin{corollary} 
If $X\oplus_p Y$ is $\delta$-average rough and $1<p\leq\infty$, then $X$ and $Y$ are $\delta$-average rough.
\end{corollary}

\section{Octahedrality and strong diameter two properties of absolute sums}\label{sec: OH and SD2P}

In this section, we characterize those absolute norms for which the direct sum of two octahedral Banach spaces is octahedral. In fact, there are many such norms besides the $\ell_1$- and $\ell_\infty$-norm. Since octahedrality and the strong diameter 2 property are dually connected, it follows that there are many absolute norms which preserve the strong diameter 2 property. In order to present these characterizations we will introduce the notions of positive octahedrality and the positive strong diameter 2 property. We end this section by proving that, similarly to the Daugavet property, among all of the absolute norms the diametral strong diameter 2 property is stable only for $\ell_1$- and $\ell_\infty$-sums.

We begin by recalling the following equivalent formulation of octahedrality from \cite{haller_duality_2015}.

\begin{proposition}[see {\cite[Proposition~2.2]{haller_duality_2015}}]\label{prop: reformulations of oct}
Let $X$ be a Banach space.
The following assertions are equivalent:
\begin{itemize}
\item[{\rm(i)}]
$X$ is octahedral;
\item[{\rm(ii)}]
whenever $n\in\N$, $x_1,\dotsc,x_n\in S_X$, and $\eps>\nobreak0$, there is~a $y\in S_X$ such that
\[
\|x_i+y\|\geq2-\eps\quad\text{for all $i\in\{1,\dotsc,n\}$.}
\]
\end{itemize}
\end{proposition}

%If $a\geq 0$ and $b\geq 0$, then we say that the element $(a,b)\in \R^2$ is \emph{positive}.

\begin{definition}
An element $(a,b)\in \R^2$ is \emph{positive} if $a\geq 0$ and $b\geq 0$. Let $N$ be an absolute normalized norm on $\R^2$. We say that $(\R^2,N)$ is \emph{positively octahedral} if whenever $n\in \mathbb{N}$ and positive $(a_1,b_1),\dots, (a_n,b_n)\in S_{(\R^2, N)}$ there is a positive $(c,d)\in S_{(\R^2,N)}$ such that
\[
N\big((a_i,b_i)+(c,d)\big)=2 \quad \text{ for all $i\in\{1,\dots,n\}$}.
\]
\end{definition}

\begin{remark*}
Note that $(\R^2, N)$ is positively octahedral if and only if there is a $(c,d)\in S_{(\R^2,N)}$ such that
\[
N\big((1,0)+(c,d)\big)=2 \qquad\text{and}\qquad N\big((0,1)+(c,d)\big)=2.
\]
\end{remark*}

%Since the unit ball of a space with an absolute norm is completely characterized by the first quadrant, we can illustrate the positively octahedral ones as follows.
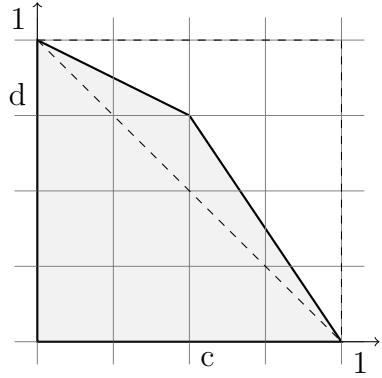
\begin{figure}[h!]
\begin{center}
\begin{tikzpicture}
\draw[thick, fill=gray!10] (0,0) -- (0,4) -- (2,3) -- (4,0) -- cycle;
\draw[step=1cm,gray,very thin] (-0.3,-0.3) grid (4.3,4.3);
\draw[->] (0,0) -- (4.5,0);
\draw[->] (0,0) -- (0,4.5);
\draw[dashed](0,4) -- (4,4) -- (4,0);
\draw[dashed](0,4) -- (4,0);
\node [below right] at (2,0) {c};
\node [above left] at (0,3) {d};
\node [below right] at (4,0) {1};
\node [above left] at (0,4) {1};
\end{tikzpicture}
\end{center}
\caption{First quadrant of the unit ball of a positively octahedral $(\R^2,N)$.}
\end{figure}

\newpage

\begin{theorem}\label{thm: absolute sum is OH}
Let $X$ and $Y$ be octahedral Banach spaces and $N$ an absolute normalized norm on $\R^2$. Then $X\oplus_N Y$ is octahedral if and only if $(\R^2,N)$ is positively octahedral.
\end{theorem}
\begin{proof}
\emph{Necessity.} Assume that $X\oplus_N Y$ is octahedral. Let $\varepsilon>0$ and positive $(a_1,b_1),\dotsc,(a_n,b_n)\in S_{(\R^2, N)}$. We will show that there is a positive $(c,d)\in S_{(\R^2, N)}$ such that
\[
N\Big((a_i,b_i)+(c,d)\Big)>2-\varepsilon \quad \text{ for all $i\in\{1,\dots,n\}$}.
\]
Let $x_i\in X$ and $y_i\in Y$ be such that $\|x_i\|=a_i$ and $\|y_i\|=b_i$. Since $X\oplus_N Y$ is octahedral, there exists a $(u,v)\in S_{X\oplus_N Y}$ such that $\|(u,v)\|_N=1$ and 
\[
\|(x_i,y_i)+(u,v)\|_N>2-\varepsilon \quad \text{ for all $i\in\{1,\dots,n\}$}.
\]
Take $c=\|u\|$ and $d=\|v\|$. Then for every $i$
\begin{align*}
N\Big((a_i,b_i)+(c,d)\Big)&=N(a_i+c,b_i+d)\\
&=N(\|x_i\|+\|u\|, \|y_i\|+\|v\|)\\
&\geq N(\|x_i+u\|, \|y_i+v\|)\\
&> 2-\varepsilon.
\end{align*}
\emph{Sufficiency.} Assume that $(\R^2, N)$ is positively octahedral. Let $(x_1, y_1),\dotsc,
(x_n,y_n)\in X\oplus_N Y$ be with norm one and $\varepsilon>0$. We will show that there is a $(u,v)\in X\oplus_N Y$ with norm one such that
\[
\|(x_i,y_i)+(u,v)\|_N\geq (1-\varepsilon)(2-\varepsilon) \quad \text{ for all $i\in\{1,\dots,n\}$}.
\]
Since $(\R^2, N)$ is positively octahedral, there is a positive $(c,d)\in S_{(\R^2, N)}$ such that 
\[
N(\|x_i\|+c, \|y_i\|+d)\geq 2-\varepsilon\quad \text{ for all $i\in\{1,\dots,n\}$}.
\]
Since $X$ and $Y$ are octahedral, there are $x\in S_X$ and $y\in S_Y$ such that 
\[
\|x_i+tx\|\geq (1-\varepsilon)(\|x_i\|+t)\quad \text{ for all $t\geq 0$}
\]
and
\[
\|y_i+ty\|\geq (1-\varepsilon)(\|x_i\|+t)\quad \text{ for all $t\geq 0$}.
\]
Take $u=cx$ and $v=dy$. It follows that $\|(u,v)\|_N=1$ and
\begin{align*}
\|(x_i,y_i)+(u,v)\|_N&=N(\|x_i+cx\|, \|y_i+dy\|)\\
&\geq (1-\varepsilon)N\Big(\|x_i\|+c, \|y_i\|+d\Big)\\
&\geq(1-\varepsilon)(2-\varepsilon).
\end{align*}
\end{proof}

Recall (see \cite{kadets_thickness_2011}) that a Banach space $X$ has the \emph{almost Daugavet property} if there is a 1-norming subspace $Y$ of $\Xast$ such that 
\[
\|Id+T\|=1+\|T\|
\]
holds true for every rank-one operator $T\colon X\to X$ of the form $T=\yast\otimes x$, where $x\in X$ and $\yast\in Y$. This definition is a generalization of the well-known Daugavet property, where $Y=\Xast$. In \cite[Propositions~2.1 and 2.2]{lücking_subspaces_2011}, it is shown that if $X$ and $Y$ are separable Banach spaces with the almost Daugavet property, then $X\oplus_1 Y$ and $X\oplus_\infty Y$ have the almost Daugavet property too. Since the almost Daugavet property and octahedrality coincide for separable Banach spaces (see \cite[Theorem~1.1]{kadets_thickness_2011}), we immediately get from Theorem~\ref{thm: absolute sum is OH} the following stability result for almost Daugavet spaces.

\begin{corollary}\label{cor: almost Daugavet}
Let $X$ and $Y$ be separable Banach spaces with the almost Daugavet property and $N$ an absolute normalized norm on $\R^2$. Then $X\oplus_N Y$ has the almost Daugavet property if and only if $(\R^2,N)$ is positively octahedral.
\end{corollary}

%In this section, we characterize the absolute normalized norms $N$ on $\mathbb R^2$ for which the space $X\oplus_N Y$ has the strong diameter 2 property whenever $X$ and $Y$ have strong diameter 2 property.

In order to characterize those absolute norms which preserve the strong diameter 2 property, we introduce the following notion.

\begin{definition}
Let $N$ be an absolute normalized norm on $\R^2$. We say that $\R^2$ has the \emph{positive strong diameter 2 property} if whenever $n\in \mathbb{N}$, 
positive $f_1,\dots,f_n\in S_{(\R^2, N^\ast)}$, $\alpha_1,\dots,\alpha_n>0$, and $\lambda_1,\dots, \lambda_n\geq 0$ with $\sum_{i=1}^n \lambda_i=1$ there are positive $(a_i,b_i)\in S(B_{(\R^2, N)}, f_i, \alpha_i)$ such that
\[
N\Big(\sum_{i=1}^n \lambda_i(a_i,b_i)\Big)=1. 
\]
\end{definition}
\begin{remark*}
Note that $(\mathbb R^2,N)$ has the positive strong diameter 2 property if and only if there are $a,b\geq 0$ such that $N(a,1)=N(1,b)=1$ and $N\Big(\dfrac{1}{2}(a,1)+\dfrac{1}{2}(1,b)\Big)=1$.
\end{remark*}
\begin{figure}[h!]
\begin{center}
\begin{tikzpicture}
\draw[thick, fill=gray!10] (0,0) -- (4,0) -- (4,1) -- (2,4) -- (0,4) -- cycle;
\draw[step=1cm,gray,very thin] (-0.3,-0.3) grid (4.3,4.3);
\draw[->] (0,0) -- (4.5,0);
\draw[->] (0,0) -- (0,4.5);
\draw[dashed](4,1) -- (4,4) -- (2,4);
\draw[dashed](0,4) -- (4,0);
\node [below right] at (2,0) {a};
\node [above left] at (0,1) {b};
\node [below right] at (4,0) {1};
\node [above left] at (0,4) {1};
\end{tikzpicture}
\end{center}
\caption{First quadrant of the unit ball of $(\R^2,N)$ with the positive strong diameter 2 property.}
\end{figure}
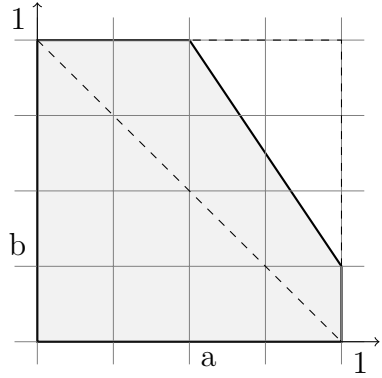

\newpage

\begin{proposition}\label{prop: pos SD2P = pos OH}
	Let $N$ be an absolute normalized norm on $\R^{2}$. The space $(\R^{2},N)$ has the positive strong diameter 2 property if and only if $(\R^{2},N^{*})$ is positively octahedral.
\end{proposition}
\begin{proof}
\emph{Necessity.} 
Assume that $(\mathbb R^2, N)$ has the positive strong diameter 2 property. So there are $a,b\geq 0$ such that $N(a,1)=N(1,b)=1$ and
\[
N\Big(\frac12(a,1)+\frac12(1,b)\Big)=1.
\]
Let $c,d\geq 0$ be such that $N^\ast(c,d)=1$ and 
\[
(c,d)\Big(\frac12(a,1)+\frac12(1,b)\Big)=1.
\]
It implies that $(c,d)(a,1)=(c,d)(1,b)=1$. Hence
\[
N^\ast((1,0)+(c,d))=((1,0)+(c,d))(1,b)=2
\]
and
\[
N^\ast((0,1)+(c,d))=((0,1)+(c,d))(a,1)=2.
\]
Therefore $(\mathbb R^2,N^\ast)$ is positively octahedral.

\emph{Sufficiency.} 
Assume now that $(\mathbb R^2,N^\ast)$ is positively octahedral. So there exist $c,d\geq 0$ such that $N^\ast(c,d)=1$ and 
\[
N^\ast((1,0)+(c,d))=2\qquad \text{and}\qquad N^\ast((0,1)+(c,d))=2.
\]
Let $a,b,x,y\geq0$ be such that $N(a,y)=1$, $N(x,b)=1$, 
\[
((1,0)+(c,d))(x,b)=2,
\]
and
\[
((0,1)+(c,d))(a,y)=2.
\]

It follows that $(1,0)(x,b)=1$ and $(0,1)(a,y)=1$ which means that $x=y=1$. Hence
\[
N\Big(\frac12(a,1)+\frac12(1,b)\Big)=(c,d)\Big(\frac12(a,1)+\frac12(1,b)\Big)=\frac12+\frac12=1.
\]
Therefore $(\mathbb \R^2,N)$ has the positive strong diameter 2 property.
\end{proof}

The duality between the strong diameter 2 property and octahedrality, Theorem~\ref{thm: absolute sum is OH}, and Proposition~\ref{prop: pos SD2P = pos OH} yield the following result, however, we prefer to give its direct proof.

\begin{theorem}\label{thm: absolute sum has SD2P}
Let $X$ and $Y$ be Banach spaces with the strong diameter 2 property and $N$ an absolute normalized norm on $\R^2$. Then $X\oplus_N Y$ has the strong diameter 2 property if and only if $(\R^2,N)$ has the positive strong diameter 2 property.
\end{theorem}
\begin{proof}

\emph{Necessity.} Assume that $X\oplus_N Y$ has the strong diameter 2 property. We will show that $(\R^2, N)$ has the positive strong diameter 2 property. Let $(c_1,d_1),\dots, (c_n,d_n)$ be positive elements in $S_{(\R^2, N^\ast)}$, $\alpha_1,\dots,\alpha_n>0$, $\lambda_i>0$ with $\sum_{i=1}^n\lambda_i=1$, and $\varepsilon>0$. We will show that there are positive $(a_i,b_i)\in B_{(\R^2, N)}$ such that $c_ia_i+d_ib_i>1-\alpha_i$ and $N(\sum_{i=1}^n\lambda_i(a_i,b_i))>1-\varepsilon$.

Let $(\xast_i, \yast_i)\in S_{X^\ast\oplus_{N^\ast} Y^\ast}$ be such that $\|\xast_i\|=c_i$ and $\|\yast_i\|=d_i$ for every $i$. Since $X\oplus_N Y$ has the strong diameter 2 property, there are 
\[
(x_i,y_i)\in S(B_{X\oplus_N Y}, (\xast_i, \yast_i),\alpha_i)
\]
such that $\|\sum_{i=1}^n\lambda_i(x_i,y_i)\|_N\geq 1-\varepsilon$.

Take $(a_i,b_i)=(\|x_i\|,\|y_i\|)$. Then $c_ia_i+d_ib_i>1-\alpha_i$, because
\[
c_ia_i+d_ib_i=\|\xast_i\|\|x_i\|+\|\yast_i\|\|y_i\|\geq \xast_i(x_i)+\yast_i(y_i)>1-\alpha_i
\]
and
\begin{align*}
N\Big(\sum_{i=1}^n\lambda_i(a_i,b_i)\Big)&=N\Big(\sum_{i=1}^n\lambda_i\|x_i\|,\sum_{i=1}^n\lambda_i\|y_i\|\Big)\\
&\geq N\Big(\|\sum_{i=1}^n\lambda_ix_i\|,\|\sum_{i=1}^n\lambda_iy_i\|\Big)\\
&=\|\sum_{i=1}^n\lambda_i(x_i,y_i)\|_N\geq 1-\varepsilon.
\end{align*}

\emph{Sufficiency.} We use an idea from \cite{haller_diametral_2016}. Assume that $(\R^2, N)$ has the positive strong diameter 2 property. Let $S_1,\dots, S_n$ be slices of $B_{X\oplus_N Y}$ defined by norm one functionals $(\xast_i, \yast_i)$ and scalars $\alpha_i>0$. Let $\lambda_i>0$ be such that $\sum_{i=1}^n\lambda_i=1$. We will show that the diameter of $\sum_{i=1}^{n}\lambda_{i}S_{i}$ is 2.

Let $\varepsilon>0$. Consider the slices $S^X_i=S(B_X,\frac{\xast_i}{\|\xast_i\|},\frac{\alpha_i}{2})$ and $S^Y_i=S(B_Y,\frac{\yast_i}{\|\yast_i\|},\frac{\alpha_i}{2})$ (If $\xast_i=0$, then $S^X_i=B_X$ and if $\yast_i=0$, then $S^Y_i=B_Y$).

Since $(\R^2, N)$ has the positive strong diameter 2 property, there are positive $(a_i,b_i)\in S(B_{(\R^2, N)}, (\|\xast_i\|,\|\yast_i\|),\delta)$ such that $N\Big(\sum_{i=1}^n\lambda_i(a_i,b_i)\Big)>1-\delta$, where $\delta>0$ satisfies $(1-\delta)(1-\alpha_i/2)\geq 1-\alpha_i$ for all $i\in\{1,\dots,n\}$.

It turns out that $a_iS^X_i\times b_iS^Y_i\subset S_i$. Indeed, if $x\in S^X_i$ and $y\in S^Y_i$, then
\[
\|(a_ix,b_iy)\|_N=N(a_i\|x\|,b_i\|y\|)\leq N(a_i,b_i)\leq 1
\]
and
\[
a_i\xast_i(x)+b_i\yast_i(y)>(1-\delta)(1-\frac{\alpha_i}{2})\geq 1-\alpha_i.
\]
Denote by
	\begin{displaymath}
	a=\sum_{i=1}^{n}\lambda_{i}a_{i} \qquad\text{and}\qquad b =\sum_{i=1}^{n}\lambda_{i}b_{i}.
	\end{displaymath}
Suppose that $a\neq 0$ and $b\neq 0$.
	For every $i$, denote by
	\begin{displaymath}
	\mu_{i}=\frac{\lambda_{i}a_{i}}{a} \qquad \text{and}\qquad \nu_{i}=\frac{\lambda_{i}b_{i}}{b}.
	\end{displaymath}
	As $X$ and $Y$ have the strong diameter 2 property, then there are $\widehat{x}, \widehat{u}\in \sum_{i=1}^n\mu_iS^X_i$ and $\widehat{y}, \widehat{v}\in \sum_{i=1}^n\nu_iS^Y_i$ such that $\|\widehat{x}-\widehat{u}\|\geq 2-\varepsilon$ and $\|\widehat{y}-\widehat{v}\|\geq 2-\varepsilon$.
	Take $x=a\widehat{x}$, $y=b\widehat{y}$, $u=a\widehat{u}$, and $v=b\widehat{v}$. Then $(x,y),(u,v)\in\sum_{i=1}^n\lambda_iS_i$, because $x,u\in \sum_{i=1}^n\lambda_ia_iS^X_i$ and $y,v\in \sum_{i=1}^n\lambda_ib_iS^Y_i$.
	Finally, 
	\begin{align*}
	\|(x,y)-(u,v)\|_N&=N(\|x-u\|,\|y-v\|)\\
	&\geq (2-\varepsilon)N(a,b)\\
	&>(2-\varepsilon)(1-\delta).
	\end{align*}
	Consider now the case, where $a=0$ or $b=0$. Assume that $a=0$. Since
	\begin{displaymath}
	\lbrace 0 \rbrace \times S_{i}^{Y} \subset S_{i},
	\end{displaymath}
	then
	\begin{displaymath}
	\lbrace 0 \rbrace \times \sum_{i=1}^{n} \lambda_{i}S_{i}^{Y} \subset \sum_{i=1}^{n}\lambda_{i}S_{i}.
	\end{displaymath}
	As the diameter of $ \sum_{i=1}^{n}\lambda_{i}S_{i}^{Y} $ is 2, there are $y,v\in \sum_{i=1}^{n}\lambda_{i}S_{i}^{Y}$
	such that
	\begin{displaymath}
	\Vert y-v\Vert \geq 2-\varepsilon.
	\end{displaymath}
Thus $(0,y),(0,v)\in \sum_{i=1}^{n}\lambda_{i}S_{i}$. Now we have
	\begin{align*}
	\Vert (0,y)-(0,v)\Vert_{N}&= N(0,\Vert y-v\Vert)\\ &=\Vert y-v\Vert \\
	&\geq 2-\varepsilon.
	\end{align*}
\end{proof}

We now turn our attention to investigate the stability of the diametral strong diameter 2 property. From \cite{becerra_guerrero_diametral_2015} and \cite{haller_diametral_2016}, we know that $X\oplus_\infty Y$ and $X\oplus_1 Y$ have the diametral strong diameter 2 property as soon as $X$ and $Y$ have the diametral strong diameter 2 property. We end this section by proving that there are no other absolute norms different from $\ell_1$- and $\ell_\infty$-norm which preserve the diametral strong diameter 2 property. Since the diametral strong diameter 2 property implies the strong diameter 2 property and the latter is stable only for absolute norms with the positive strong diameter 2 property, we can restrict our attention to them.

Consider an absolute normalized norm $N$ on $\mathbb R^2$, different from the $\ell_1$-norm and $\ell_\infty$-norm, such that $(\mathbb R^2,N)$ has the positive strong diameter 2 property. Thus, for some
$a,b\in[0,1)$ with $a>0$ or $b>0$, $N$ is defined by
\begin{equation}\label{eq: pos SD2P}
N(c,d)=\max\Big\{|c|,|d|,\dfrac{(1-b)|c|+(1-a)|d|}{1-ab}\Big\}\qquad \text{for all $(c,d)\in\mathbb R^2$.}
\end{equation}

\begin{proposition}\label{prop: absolute sum DSD2P}
Let $X$ and $Y$ be nontrivial Banach spaces and $N$  defined by (\ref{eq: pos SD2P}). Then $X\oplus_N Y$ does not have the diametral strong diameter 2 property.
\end{proposition}

We will use the following elementary lemma.

\begin{lemma}\label{lem: lambda}
There is a $\lambda\in(0,1)$ such that
\[
N\Big(2\lambda+(1-\lambda)a,2(1-\lambda)+\lambda b\Big)<1+N(\lambda,1-\lambda).
\]
\end{lemma}
\begin{proof}
Assume that $\lambda\in(0,1)$. Denote by \[c=2\lambda+(1-\lambda)a\qquad \text{and}\qquad d=2(1-\lambda)+\lambda b.\] It is straightforward to show directly that the condition \[N(c,d)=\dfrac{(1-b)c+(1-a)d}{1-ab}\] is equivalent to
\[
\dfrac{a}{2+a-ab}\leq\lambda\leq\dfrac{2-ab}{2+b-ab},
\]
and the condition
\[\dfrac{(1-b)c+(1-a)d}{1-ab}<1+N(\lambda,1-\lambda)\]
is equivalent to
\[\lambda<\dfrac{a}{1+a}\qquad \text{or}\qquad \lambda>\dfrac{1}{1+b}.\]
Note that
\[\dfrac{a}{2+a-ab}\leq\dfrac{a}{1+a}\leq\dfrac{1}{1+b}\leq\dfrac{2-ab}{2+b-ab},\]
where the first inequality is strict if and only if $a\not=0$, and the last inequality is strict if and only if $b\not=0$.
\end{proof}

\begin{proof}[Proof of Proposition~\ref{prop: absolute sum DSD2P}]
By using Lemma~\ref{lem: lambda}, we choose $\lambda\in(0,1)$ such that
\[
N\Big(2\lambda+(1-\lambda)a,2(1-\lambda)+\lambda b\Big)<1+N(\lambda,(1-\lambda).
\]
Denote by
\[
\delta=1+N(\lambda,1-\lambda)-N\Big(2\lambda+(1-\lambda)a,2(1-\lambda)+\lambda b\Big).
\]
Choose any $\varepsilon\in(0,\delta/2)$. Let $\alpha>0$ be such that if $(c_1,d_1), (c_2,d_2)\in \mathbb R^2$ satisfy the conditions $N(c_1,d_1),N(c_2,d_2)\leq 1$, $|c_1|>1-\alpha$, and $|d_2|>1-\alpha$, then
\[
N\Big(2\lambda+(1-\lambda)|c_2|,2(1-\lambda)+\lambda |d_1|\Big)\leq N\Big(2\lambda+(1-\lambda)a,2(1-\lambda)+\lambda b\Big)+\varepsilon.
\]
Fix any $x^\ast\in S_{X^\ast}$ and $y^\ast\in S_{Y^\ast}$. Consider the slices $S_1=S(B_{X\oplus_N Y},(x^\ast,0),\alpha)$ and $S_2=S(B_{X\oplus_N Y},(0,y^\ast),\alpha)$. Choose $x\in S_X$ and $y\in S_Y$ such that $(x,0)\in S_1$ and $(0,y)\in S_2$.
Assuming that the Banach space $X\oplus_N Y$ has the diametral strong diameter 2 property, there exist $(u_1,v_1)\in S_1$ and $(u_2,v_2)\in S_2$ such that
\begin{align*}
\tilde N:=\|\lambda(x,0)&+(1-\lambda)(0,y)-\lambda (u_1,v_1)-(1-\lambda)(u_2,v_2)\|_N\geq\\&\geq \|\lambda (x,0)+(1-\lambda)(0,y)\|_N+1-\varepsilon.
\end{align*}
Since
\begin{align*}\tilde N&=N\Big(\|\lambda x-\lambda u_1-(1-\lambda)u_2\|,\|(1-\lambda)y-\lambda v_1-(1-\lambda)v_2\|\Big)\leq\\
&\leq N\Big(2\lambda+(1-\lambda)\|u_2\|,2(1-\lambda)+\lambda\|v_1\|\Big)\leq\\
&\leq N\Big(2\lambda+(1-\lambda) a,2(1-\lambda)+\lambda b\Big)+\varepsilon=\\
&=1+N(\lambda,1-\lambda)-\delta+\varepsilon,
\end{align*}
it follows that 
\[
\|\lambda (x,0)+(1-\lambda)(0,y)\|_N+1-\varepsilon\leq 1+N(\lambda,1-\lambda)-\delta+\varepsilon,
\]
i.e., $\delta\leq 2\varepsilon$, which is a contradiction.
\end{proof}

Combining \cite[Theorem~3.8]{becerra_guerrero_diametral_2015}, \cite[Theorem]{haller_diametral_2016}, and Proposition~\ref{prop: absolute sum DSD2P}, we get the following corollary.

\begin{corollary}\label{cor: abs sum DSD2P}
If $Z=X\oplus_N Y$ has the diametral strong diameter 2 property, then either $Z=X\oplus_1 Y$ or $Z=X\oplus_\infty Y$.
\end{corollary}

\bibliographystyle{amsplain}
\footnotesize
%\bibliography{viited}

\end{document}